\documentclass[11pt]{article}
\usepackage{amsmath,amssymb,amsthm,comment,cancel,mathrsfs}
\usepackage[numbers]{natbib}

\newtheorem{theorem}{Theorem}
\newtheorem{lemma}[theorem]{Lemma}

\newtheorem{proposition}[theorem]{Proposition}
\theoremstyle{definition}
\newtheorem{problem}{Problem}

\newtheorem*{construction*}{Construction $\mathfrak{M}$}
\newtheorem*{theorem*}{The Rho Function}

\title{Incidence Gain Graphs and Generalized Quadrangles}
\author{Ryan McCulloch\footnote{Ryan McCulloch, Binghamton University, rmccullo1985@gmail.com}}
\date{February 3, 2025}

\begin{document}

\maketitle

\begin{abstract}
We demonstrate a construction method based on a gain function that is defined on the incidence graph of an incidence geometry.  Restricting to when the incidence geometry is a linear space, we show that the construction yields a generalized quadrangle provided that the gain function satisfies a certain bijective property.  Our method is valid for finite and infinite geometries.  We produce a family of generalized quadrangles by defining such a gain function on an affine plane over an arbitrary field.
\end{abstract}

{\small
\noindent
{\bf MSC2000\,:} Primary 05B25; Secondary 51E12.

\noindent
{\bf Key words\,:} gain graph, incidence graph, incidence structure, generalized quadrangle, generalized polygon, linear space, Steiner system, affine plane, affine geometry, incidence geometry, finite geometry.}

\section{Introduction and Theory}
This journey through incidence geometry passes through gain graphs, generalized polygons, linear spaces, and group actions.  We attempt to be as general as needed, with a curiosity towards the infinite case.  Our goal is to demonstrate a construction method using incidence gain graphs, and to show when and how the construction produces a generalized quadrangle.

\subsection{Definitions}

An \textit{incidence structure} is a triple $(\mathscr{P},\mathscr{B},\mathrm{I})$ where $\mathscr{P}$ is a set of points, $\mathscr{B}$ is a set of lines (or blocks), $\mathscr{P}$ and $\mathscr{B}$ are nonempty disjoint sets of objects, and $\mathrm{I} \subseteq \mathscr{P} \times\mathscr{B}$ is an incidence relation which is symmetric.  Assume that $\mathrm{I}$ is nonempty.  We write $p \  \mathrm{I} \  b$ or $b \  \mathrm{I} \ p$ to indicate that $(p,b) \in \mathrm{I}$ and we write $p \  \cancel{\mathrm{I}} \  b$ or $b \  \cancel{\mathrm{I}} \  p$ to indicate that $(p,b) \notin \mathrm{I}$.  We write $\mathscr{P}_b = \{ q \in \mathscr{P} : q \ \text{I} \ b \}$.  The $\textit{incidence graph}$ of an incidence structure $(\mathscr{P},\mathscr{B},\mathrm{I})$ is a bipartite graph $\Gamma$ with vertex set $V = \mathscr{P} \bigcup\mathscr{B}$ and edge set $E = \{ \{b,p\} : b \  \mathrm{I} \  p \}$.  We often drop the brackets and write an edge as $e = bp$.   

A $\textit{gain graph}$ is a pair $(\Gamma,\varphi)$, where $\Gamma = (V,E)$ is a graph and $\varphi$ is a $\textit{gain function}$, which is a homomorphism $\varphi : \mathscr{F}(E) \rightarrow G$ where $G$ is a group, called the \textit{gain group}, and $\mathscr{F}(E)$ is the free group on $E$.  We see that the gain function $\varphi$ is uniquely determined by the values $\varphi(e)$ for $e \in E$ the generating set.

Note that for an edge $e$, we have $\varphi(e^{-1}) = \varphi(e)^{-1}$, and it is understood that $e^{-1}$ means $e$ with its orientation reversed.  Let us declare that all edges in an incidence graph $\Gamma$ are oriented from ``line to point''. 

A gain graph of the incidence graph of an incidence structure shall be called an \textit{incidence gain graph}.

Throughout this article, we will have a gain group $G$ acting on a set $\Lambda$ on the left.  

We define \textit{conjugation of g by h} in a group by ${}^hg := h g h^{-1} $.  We have $$({}^hg) \cdot \lambda = \mu \Longleftrightarrow g \cdot (h^{-1} \cdot \lambda) = h^{-1} \cdot \mu.$$  Note that $\lambda$ is fixed by ${}^hg$ if and only if $h^{-1} \cdot \lambda$ is fixed by $g$.

A group action of $G$ on $\Lambda$ is \textit{free} if $g \cdot \lambda = \lambda$ for some $\lambda \in \Lambda$ implies that $g$ is the identity element of $G$.  A group action is \textit{transitive} if for each $\mu, \lambda \in \Lambda$ there is $g \in G$ so that $\mu = g \cdot \lambda$.  A group action is \textit{regular} if it is both free and transitive.  Note that the action of a group $G$ on itself via multiplication is a regular action.

A \textit{walk} in a graph $\Gamma$ is a sequence $$w = (u_0,e_1,u_1,e_2,\dots ,e_n,u_n)$$ where the vertices $u_{i-1}$ and $u_i$ are incident with the edge $e_i$.  For a gain graph $(\Gamma,\varphi)$, we define $$ \varphi_w := \varphi({e_n}^{\delta_n} \cdots {e_1}^{\delta_1}) = \varphi(e_n)^{\delta_n} \cdots \varphi(e_1)^{\delta_1},$$ where $\delta_i = 1$ if $e_i$ is oriented from $u_{i-1}$ to $u_i$ and $\delta_i = -1$ if $e_i$ is oriented from $u_i$ to $u_{i-1}$.  We multiply ``backwards'' because the gain group $G$ acts on a set $\Lambda$ on the left.

Given a gain graph $(\Gamma,\varphi)$ with $\Gamma = (V,E)$ and gain group $G$, and any function $f : V \rightarrow G$, we obtain another gain function ${}^f\varphi$ by defining $${}^f\varphi (e) := f(v)\varphi(e)f(u)^{-1},$$
for every edge $e = uv$ oriented from $u$ to $v$.  This procedure is known as \textit{switching} $\varphi$ \textit{by} $f$, and the function $f$ is called a \textit{switching function}.

Gain graphs and switching were introduced in \cite{Zas89} for matroids.  Our definitions may differ slightly than in \cite{Zas89} to suit our purposes.  Gain graphs are also known as voltage graphs in the context of surface embeddings of graphs.

Given $(\mathscr{P},\mathscr{B},\mathrm{I})$ and $(\mathscr{P}',\mathscr{B}',\mathrm{I}')$ incidence structures, an \textit{incidence structure isomorphism} is a pair $(g_1,g_2)$ of bijections $g_1 : \mathscr{P} \rightarrow \mathscr{P}'$ and $g_2 : \mathscr{B} \rightarrow \mathscr{B}'$ so that $p \  \mathrm{I} \  b$ if and only if $g_1(p) \  \mathrm{I}' \  g_2(b)$.

For an incidence structure $(\mathscr{P},\mathscr{B},\mathrm{I})$, a $k$\textit{-chain} is a sequence $(u_0, u_1, \dots, u_k)$ of elements $u_i \in \mathscr{P} \cup \mathscr{B}$ with the property that $u_i$ is incident with $u_{i-1}$ for all $i \in \{1, \dots,  k\}$.  We say that the \textit{distance from $u$ to $v$}, denoted $d(u,v)$, is the smallest $k$ such that there exists a $k$-chain $(u = u_0, u_1, \dots, u_k = v)$ from $u$ to $v$.  If no such $k$ exists, we say $d(u,v) = \infty$.  Note that if $d(u,v) < \infty$, we have $d(u,v)$ is even if and only if $u$ and $v$ are either both points or are both lines.

For an integer $n \geq 3$, an incidence structure $(\mathscr{P},\mathscr{B},\mathrm{I})$ is called a \textit{generalized n-gon}, if it satisfies the following two axioms.
\begin{enumerate}
 \item For each $u,v \in \mathscr{P} \cup \mathscr{B}$, we have $d(u,v) \leq n$.
\item For each $u,v \in \mathscr{P} \cup \mathscr{B}$ so that $d(u,v) = k < n$, there exists a unique $k$-chain $(u = u_0,u_1, \dots , u_k = v)$ from $u$ to $v$.
\end{enumerate}

A generalized $n$-gon is \textit{firm} if every element $\mathscr{P} \cup \mathscr{B}$ is incident with at least two other elements, and \textit{thick} if every element in $\mathscr{P} \cup \mathscr{B}$ is incident with at least three other elements. 

When $n=3$, a thick generalized triangle is a projective plane.  When $n=4$ we have a generalized quadrangle, and so on.   We refer the reader to the books \cite{PolyBook} and \cite{GQBook} on generalized polygons and generalized quadrangles.

An \textit{ovoid} of a generalized quadrangle is a set of points $O$ so that each line of the generalized quadrangle is incident with exactly one point of $O$.  Not all generalized quadrangles posses ovoids, but the ones that we construct in this article do.

An incidence structure is a \textit{linear space} if it satisfies the following three axioms.

\begin{enumerate}
\item Any two distinct points are incident with exactly one common line.
\item Each line is incident with at least two points.
\item There exist a point and a line that are not incident.
\end{enumerate}

Suppose $(\mathscr{P},\mathscr{B},\mathrm{I})$ is a linear space and the point set $\mathscr{P}$ has cardinality $v$.  If for every $b \in \mathscr{B}$, the set $\mathscr{P}_b$ has cardinality $k$ for some fixed cardinal $k$, then we call the linear space $(\mathscr{P},\mathscr{B},\mathrm{I})$ a \textit{Steiner system}, denoted $S(2,v,k)$.  We allow arbitrary (finite or infinite) cardinals for $v$ and $k$.

\subsection{General Constructions and Results}

\begin{construction*}
Let $(\Gamma,\varphi)$ be an incidence gain graph with underlying incidence structure $(\mathscr{P},\mathscr{B},\mathrm{I})$ and gain group acting on a nonempty set $\Lambda$.  Construct an incidence structure $(\mathscr{P}',\mathscr{B}',\mathrm{I}')$ as follows:

\begin{itemize}
\item $\mathscr{P}' = X \cup Y$ where $X := \{ x_p \  | \  p \in \mathscr{P} \}$ and $Y := \{ y_{b,\lambda} \  | \  b \in \mathscr{B}, \lambda \in \Lambda \}$.
\item $\mathscr{B}' = Z$ where $Z := \{ z_{p,\lambda} \  | \  p \in \mathscr{P}, \lambda \in \Lambda \}$.
\item $\mathrm{I}'$ is defined so that $x_p \  \mathrm{I}' \  z_{p,\lambda}$ for all $\lambda$ and $y_{b,\lambda} \  \mathrm{I}' \  z_{p,\mu}$ when $b \  \mathrm{I} \  p$ and $\mu = \varphi(e) \cdot \lambda$ for the edge $e = bp$ of the graph $\Gamma$.
\end{itemize}
\end{construction*}

When useful, for an incidence gain graph $(\Gamma,\varphi)$, we will use the notation $\mathfrak{M}(\Gamma,\varphi)$ to indicate the incidence structure from Construction $\mathfrak{M}$.  We use the notation of $x_p$ and $y_{b,\lambda}$ for the ``point'' and ``line'' type points of $\mathfrak{M}(\Gamma,\varphi)$ and the notation of $z_{p,\lambda}$ for the lines.  

The next proposition describes how switching induces an isomorphism of the incidence structures defined by Construction $\mathfrak{M}$.

\begin{proposition}
Let $(\Gamma,\varphi)$ be an incidence gain graph with underlying incidence structure $(\mathscr{P},\mathscr{B},\mathrm{I})$ and gain group acting on a nonempty set $\Lambda$.  Let $f$ be a switching function and let $\mathfrak{M}(\Gamma,\varphi) = (\mathscr{P}',\mathscr{B}', \mathrm{I}')$ and $\mathfrak{M}(\Gamma,{}^f\varphi) = (\mathscr{P}',\mathscr{B}',\mathrm{I}'_f)$.  

\begin{itemize}
\item Define a function $g_1: \mathscr{P}' \rightarrow \mathscr{P}'$ by $g_1(x_p) = x_p$ and $g_1(y_{b,\lambda}) = y_{b,\mu}$ where $\mu = f(b) \cdot \lambda$.
\item Define a function $g_2: \mathscr{B}' \rightarrow \mathscr{B}'$ by $g_2(z_{p,\lambda}) = z_{p,\mu}$ where $\mu = f(p) \cdot \lambda$.
\end{itemize}

The pair $(g_1,g_2)$ is an incidence structure isomorphism from $\mathfrak{M}(\Gamma,\varphi)$ to $\mathfrak{M}(\Gamma,{}^f\varphi)$.
\end{proposition}

\begin{proof}
It is clear that $g_1$ and $g_2$ are both bijections.  Also, it is clear that we need only to verify the condition on incidence for the points of type $y_{b, \lambda}$.  Let $b \in \mathscr{B}$ and $p \in \mathscr{P}$ with $b \  \mathrm{I} \  p$, and let $e = bp$ be the edge of $\Gamma$.  We have
$$y_{b, \lambda} \  \mathrm{I}' \  z_{p,\mu} \Longleftrightarrow \mu = \varphi(e) \cdot \lambda.$$

Now $g_1(y_{b, \lambda}) = y_{b,\lambda_1}$ where $\lambda_1 = f(b) \cdot \lambda$, and $g_2(z_{p, \mu}) = z_{p,\mu_1}$ where $\mu_1 = f(p) \cdot \mu$.  Hence
\begin{gather*}
\varphi(e) \cdot \lambda = \mu \Longleftrightarrow \\
f(p)\varphi(e) \cdot \lambda = \mu_1 \Longleftrightarrow \\ 
f(p)\varphi(e) f(b)^{-1} \cdot \lambda_1 = \mu_1 \Longleftrightarrow \\ 
g_1(y_{b, \lambda}) \  \mathrm{I}'_f \  g_2(z_{p, \mu}).
\qedhere
\end{gather*}
\end{proof}

We omit the proof of the next proposition, as it follows easily from Construction $\mathfrak{M}$.  We consider chains in $\mathfrak{M}(\Gamma,\varphi)$ having points only of the ``line'' type.

\begin{proposition}
Let $(\Gamma,\varphi)$ be an incidence gain graph with underlying incidence structure $(\mathscr{P},\mathscr{B},\mathrm{I})$ and gain group acting on a nonempty set $\Lambda$.  Let $u_0,\dots,u_k \in \mathscr{P} \cup \mathscr{B}$ for some integer $k \geq 1$.

We have $(u_0,\dots,u_k)$ is a $k$-chain in $(\mathscr{P},\mathscr{B},\mathrm{I})$ if and only if \\ $(m_0,\dots,m_k)$ is a $k$-chain in $\mathfrak{M}(\Gamma,\varphi)$, where for each $i \in \{ 0,\dots,k \}$, $m_i = y_{u_i,\lambda_i}$ for some $\lambda_i \in \Lambda$ if $u_i \in \mathscr{B}$ and $m_i = z_{u_i,\lambda_i}$ for some $\lambda_i \in \Lambda$ if $u_i \in \mathscr{P}$.  Furthermore, if we have such $k$-chains, then $\lambda_k = \varphi_{w} \cdot \lambda_0$, where $$w = (u_0,e_1,u_1,e_2,\dots,u_{k-1},e_{k},u_{k})$$ is the corresponding walk in $\Gamma$.
\end{proposition}

From now on we assume that the underlying incidence structure is a linear space.

\begin{theorem*}
Let $(\Gamma,\varphi)$ be an incidence gain graph with underlying incidence structure $(\mathscr{P},\mathscr{B},\mathrm{I})$ a linear space, and gain group $G$ acting on a nonempty set $\Lambda$.  For each pair $(b,p)$ so that $b \in \mathscr{B}$ and $p \in \mathscr{P}$ with $b \ \cancel{\mathrm{I}} \ p$, we define a function
\begin{gather*}
\rho_{b,p} \  : \  \mathscr{P}_b \rightarrow G \\
\rho_{b,p}(q) = \varphi_w,
\end{gather*}
where $w = (b,e_1,q,e_2,b',e_3,p)$ is the walk in $\Gamma$ with $b'$ the unique line incident with $p$ and $q$.

For each triple $(b,p,\lambda)$ so that $b \in \mathscr{B}$ and $p \in \mathscr{P}$ with $b \ \cancel{\mathrm{I}} \ p$ and $\lambda \in \Lambda$, we define a function
\begin{gather*}
\rho_{b,p,\lambda} \  : \  \mathscr{P}_b \rightarrow \Lambda \\
\rho_{b,p,\lambda}(q) = \rho_{b,p}(q) \cdot \lambda.
\end{gather*}
\end{theorem*}

\begin{lemma}
Let $(\Gamma,\varphi)$ be an incidence gain graph with underlying incidence structure $(\mathscr{P},\mathscr{B},\mathrm{I})$ a linear space, and gain group acting on a nonempty set $\Lambda$.  Let $p,q \in \mathscr{P}$, $b \in \mathscr{B}$, and $\lambda, \mu \in \Lambda$.  
\begin{enumerate}
\item If $p=q$, then $d(x_p,z_{q,\lambda}) = 1$ and $(x_p,z_{q,\lambda})$ is the unique $1$-chain in $\mathfrak{M}(\Gamma,\varphi)$ from $x_p$ to $z_{q,\lambda}$.
\item If $p\neq q$, then $d(x_p,z_{q,\lambda}) = 3$ and there is a unique $3$-chain in $\mathfrak{M}(\Gamma,\varphi)$ from $x_p$ to $z_{q,\lambda}$.
\item If $b \  \mathrm{I} \  p$ and $\mu = \varphi(bp) \cdot \lambda$, then $d(y_{b,\lambda},z_{p,\mu}) = 1$ and $(y_{b,\lambda},z_{p,\mu})$ is the unique $1$-chain in $\mathfrak{M}(\Gamma,\varphi)$ from $y_{b,\lambda}$ to $z_{p,\mu}$.
\item If $b \  \mathrm{I} \  p$ and $\mu \neq \varphi(bp) \cdot \lambda$, then $d(y_{b,\lambda},z_{p,\mu}) = 3$ and there is a unique $3$-chain in $\mathfrak{M}(\Gamma,\varphi)$ from $y_{b,\lambda}$ to $z_{p,\mu}$.
\end{enumerate}
\end{lemma}

\begin{proof}
Items 1 and 3 are clear.

\textit{Item 2.}  We have the $3$-chain $(x_p, z_{p,\lambda_1}, y_{b',\lambda_2}, z_{q,\lambda})$, where $b'$ is the unique line incident with $p$ and $q$, and $\lambda_2 = \varphi(b'q)^{-1} \cdot \lambda$ and $\lambda_1 = \varphi(b'p) \cdot \lambda_2$.  This chain is unique since the second term in a $3$-chain from $x_p$ to $z_{q,\lambda}$ must be a line with first index $p$, and the third term must be a point of ``line'' type.  The $b'$ is uniquely determined as are $\lambda_1$ and $\lambda_2$.

\textit{Item 4.}  We have the $3$-chain $(y_{b,\lambda}, z_{p,\lambda_1}, x_p, z_{p,\mu})$, where $\lambda_1 = \varphi(bp) \cdot \lambda$.  We argue that this $3$-chain is unique.  Consider any $3$-chain from $y_{b,\lambda}$ to $z_{p,\mu}$.  If the second term in such $3$-chain has first index $p'$ for some $p \neq p'$, then the third term would have to be a point of ``line'' type.  But the first index on this ``line'' type point would have to be $b$, since $b$ is the unique line incident with $p$ and $p'$.  So we have a $3$-chain of the form $(y_{b,\lambda}, z_{p',\lambda_1}, y_{b,\lambda_2}, z_{p,\mu})$.  But then $\lambda = \lambda_2$, and the first and third terms in the $3$-chain are equal, a contradiction.

So the second term in such $3$-chain has first index $p$.  Suppose that the third term is a ``line'' type point.  So we have a $3$-chain of the form \\ $(y_{b,\lambda}, z_{p,\lambda_1}, y_{b',\lambda_2}, z_{p,\mu})$ for some $b'$.  But then $\mu = \lambda_1$, and the second and fourth terms in the $3$-chain are equal, a contradiction.
\end{proof}

\begin{lemma}
Let $(\Gamma,\varphi)$ be an incidence gain graph with underlying incidence structure $(\mathscr{P},\mathscr{B},\mathrm{I})$ a linear space, and gain group acting on a nonempty set $\Lambda$.  Let $\mathfrak{M}(\Gamma,\varphi) = (\mathscr{P}',\mathscr{B}',\mathrm{I}')$.  If $u,v \in \mathscr{P}' \cup \mathscr{B}'$ are either both points or are both lines, and if $d(u,v) \leq 2$, then there is a unique $2$-chain in $\mathfrak{M}(\Gamma,\varphi)$ from $u$ to $v$.
\end{lemma}

\begin{proof}
\textit{Case 1.}  Suppose $u = z_{p,\lambda}$ and $v = z_{q,\mu}$ for some $p,q \in \mathscr{P}$ and $\lambda,\mu \in \Lambda$.  

If $p=q$ and $\lambda \neq \mu$, then $(z_{p,\lambda}, x_p, z_{q,\mu})$ is a $2$-chain, and if we have a $2$-chain from $z_{p,\lambda}$ to $z_{q,\mu}$ where the second term is a ``line'' type point, then $\lambda = \mu$, a contradiction.  And so this is the unique $2$-chain from $u$ to $v$.  

If $p \neq q$, then we have a $2$-chain $(z_{p,\lambda}, y_{b,\lambda_1}, z_{q,\mu})$ where $b$ is the unique line incident with $p$ and $q$, and where $\lambda_1 = \varphi(bp)^{-1} \cdot \lambda = \varphi(bq)^{-1} \cdot \mu$, and so $\lambda = \mu$ and this is the unique $2$-chain from $u$ to $v$.  

\textit{Case 2.}  If $u = x_p$ and $v =x_q$ for some $p,q \in \mathscr{P}$, then $p=q$ and $d(u,v) = 0$.

\textit{Case 3.} Suppose $u = x_p$ and $v = y_{b,\lambda}$ for some $p \in \mathscr{P}$, $b \in \mathscr{B}$, and $\lambda \in \Lambda$.  Then $(x_p, z_{p,\mu}, y_{b,\lambda})$ is the unique $2$-chain from $u$ to $v$, where $\mu = \varphi(bp) \cdot \lambda$.

\textit{Case 4.} Suppose $u = y_{b,\lambda}$ and $v = y_{b',\mu}$ for some $b,b' \in \mathscr{B}$ and $\lambda, \mu \in \Lambda$.  We have a $2$-chain $(y_{b,\lambda},z_{p,\lambda_1}, y_{b',\mu})$ where $b$ and $b'$ are incident with a common point $p$, and $\lambda_1 = \varphi(bp) \cdot \lambda = \varphi(b'p) \cdot \mu$.  This $2$-chain is unique as $(\mathscr{P},\mathscr{B},\mathrm{I})$ a linear space and hence a common point $p$ is unique.
\end{proof}

We now come to the main theorem of this section, which describes the situation when $\mathfrak{M}(\Gamma,\varphi)$ is a generalized quadrangle.

\begin{theorem}
Let $(\Gamma,\varphi)$ be an incidence gain graph with underlying incidence structure $(\mathscr{P},\mathscr{B},\mathrm{I})$ a linear space, and gain group acting on a nonempty set $\Lambda$.  We have $\mathfrak{M}(\Gamma,\varphi)$ is a generalized quadrangle if and only if for every $p \in \mathscr{P}$ and $b \in \mathscr{B}$ with $p \  \cancel{\mathrm{I}} \  b$, and for every $\lambda \in \Lambda$, the rho function $\rho_{b,p,\lambda}$ is bijective.  When $\mathfrak{M}(\Gamma,\varphi)$ is a generalized quadrangle we have the following.
\begin{enumerate}
\item The set $X = \{ x_p \  | \  p \in \mathscr{P} \}$ is an ovoid of $\mathfrak{M}(\Gamma,\varphi)$.
\item For every $b \in \mathscr{B}$, the set $\mathscr{P}_b$ has the same cardinality as $\Lambda$.  Hence $(\mathscr{P},\mathscr{B},\mathrm{I})$ is a Steiner system.
\item If $\mathscr{P}$ is finite, then $\Lambda$ is finite, and $|\mathscr{P}| = v \geq 3$, $|\Lambda| = k \geq 2$, and $\mathfrak{M}(\Gamma,\varphi)$ has that every line is incident with exactly $1+s$ points and every point is incident with exactly $1+t$ lines, where $\displaystyle s=\frac{v-1}{k-1}$ and $t=k-1$.
\end{enumerate}
\end{theorem}

\begin{proof}
Suppose first that for every $p \in \mathscr{P}$ and $b \in \mathscr{B}$ with $p \  \cancel{\mathrm{I}} \  b$, and for every $\lambda \in \Lambda$, the rho function $\rho_{b,p,\lambda}$ is bijective.  We work to prove that $\mathfrak{M}(\Gamma,\varphi)$ is a generalized quadrangle (a generalized $n$-gon for $n=4$).

Lemma 3 establishes that if $m_1$ is a point and $m_2$ is a line of $\mathfrak{M}(\Gamma,\varphi)$, then $d(m_1,m_2) \leq 3$, with the exception of one case.  We will take care of the missing case at the end of the proof of this direction.  Let us assume for the moment that $d(m_1,m_2) \leq 3$ whenever $m_1$ is a point and $m_2$ is a line of $\mathfrak{M}(\Gamma,\varphi)$.  There is a simple argument to show that the distance between any objects of the same type is less than or equal to four.  Suppose $n_1,n_2$ are either both points or are both lines of $\mathfrak{M}(\Gamma,\varphi)$.  Let $n'$ be an object of $\mathfrak{M}(\Gamma,\varphi)$ that is incident with $n_2$ (by Construction $\mathfrak{M}$ we have no isolated objects in $\mathfrak{M}(\Gamma,\varphi)$).  Then $d(n_1,n') \leq 3$, and so $d(n_1,n_2) \leq 4$.

And so pending this missing case, we have $d(u,v) \leq 4$ for all objects $u,v$ of $\mathfrak{M}(\Gamma,\varphi)$.  Also, pending the missing case of Lemma 3, we have by Lemma 3 and Lemma 4 that if $d(u,v) < 4$ then there is a unique chain from $u$ to $v$.

We are left to establish the missing case.  We show that if $b \in \mathscr{B}$ and $p \in \mathscr{P}$ and $b \  \cancel{\mathrm{I}} \  p$, and if $\lambda, \mu \in \Lambda$, then $d(y_{b,\lambda},z_{p,\mu}) = 3$ and that there is a unique $3$-chain in $\mathfrak{M}(\Gamma,\varphi)$ from $y_{b,\lambda}$ to $z_{p,\mu}$.

Consider the rho function $\rho_{b,p,\lambda}$.  Since $\rho_{b,p,\lambda}$ is surjective, there exists $q \ \mathrm{I} \ b$ so that $\rho_{b,p,\lambda}(q) = \mu$.  In other words, if $b'$ is the unique line incident with $p$ and $q$, and if $w = (b,e_1,q,e_2,b',e_2,p)$ is the walk in $\Gamma$, then $$\mu = \varphi_w \cdot \lambda = \varphi(e_3)\varphi(e_2)^{-1} \varphi(e_1) \cdot \lambda.$$

Hence we have the $3$-chain $$(y_{b,\lambda}, z_{q,\lambda_1}, y_{b',\lambda_2}, z_{p,\mu}),$$ where $\lambda_1 = \varphi(e_1) \cdot \lambda$, $\lambda_2 = \varphi(e_2)^{-1} \cdot \lambda_1$, and $\mu = \varphi(e_3) \cdot \lambda_2 = \varphi_w \cdot \lambda$.

We now argue that this is the unique $3$-chain from $y_{b,\lambda}$ to $z_{p,\mu}$.  Note that any $3$-chain from $y_{b,\lambda}$ to $z_{p,\mu}$ has that the second term has first index $q'$ for some $q' \  \mathrm{I} \  b$, and the third term has first index $b''$ where $b''$ is the unique line incident with $p$ and $q'$.  If $$(y_{b,\lambda}, z_{q',\lambda'_1}, y_{b'',\lambda'_2}, z_{p,\mu})$$ 
is any $3$-chain from $y_{b,\lambda}$ to $z_{p,\mu}$, then since $\rho_{b,p,\lambda}$ is injective, and since $\mu = \varphi_w \cdot \lambda = \varphi_{w'} \cdot \lambda$ where $w'$ is the walk in $\Gamma$ associated with $q'$ and $b''$, it follows that $q = q'$.  And so $b' = b''$, and so $\lambda_1 = \lambda_1'$ and $\lambda_2 = \lambda_2'$.  Hence the $3$-chain  $$(y_{b,\lambda}, z_{q,\lambda_1}, y_{b',\lambda_2}, z_{p,\mu})$$
is the unique one from $y_{b,\lambda}$ to $z_{p,\mu}$.  Thus we have established that $\mathfrak{M}(\Gamma,\varphi)$ is a generalized quadrangle.

Conversely, suppose $\mathfrak{M}(\Gamma,\varphi)$ is a generalized quadrangle.  Let $p \in \mathscr{P}$ and $b \in \mathscr{B}$ with $p \  \cancel{\mathrm{I}} \  b$, and let $\lambda \in \Lambda$.  We argue that $\rho_{b,p,\lambda}$ is bijective.  

Let $\mu \in \Lambda$ be arbitrary.  Since $\mathfrak{M}(\Gamma,\varphi)$ is a generalized quadrangle, $d(y_{b,\lambda},z_{p,\mu}) = 3$ (it cannot be one since $b \  \cancel{\mathrm{I}} \  p$).  And so we have a $3$-chain $(y_{b,\lambda},z_{q,\lambda_1},y_{b',\lambda_2},z_{p,\mu})$.  Hence there exists $q \ \mathrm{I} \ b$ so that $\rho_{b,p,\lambda}(q) = \mu$.  And since $\mathfrak{M}(\Gamma,\varphi)$ is a generalized quadrangle, this is the unique chain from $y_{b,\lambda}$ to $z_{p,\mu}$, and hence there is a unique $q \ \mathrm{I} \ b$, so that $\rho_{b,p,\lambda}(q) = \mu$.  Thus $\rho_{b,p,\lambda}$ is bijective.
 
Suppose $\mathfrak{M}(\Gamma,\varphi)$ is a generalized quadrangle.

\textit{Item 1.}  It is clear from Construction $\mathfrak{M}$ that the set $X$ is an ovoid of $\mathfrak{M}(\Gamma,\varphi)$.

\textit{Item 2.}  Since for every $b \in \mathscr{B}$, $\rho_{b,p,\lambda}$ is a bijection from $\mathscr{P}_b$ to $\Lambda$, it follows that for every $b \in \mathscr{B}$, $\mathscr{P}_b$ and $\Lambda$ have the same cardinality.

\textit{Item 3.}  As $(\mathscr{P},\mathscr{B},\mathrm{I})$ is a linear space, we have a point and a line not incident, and every line is incident with at least two points, and so $|\mathscr{P}| \geq 3$.  Since $\mathscr{P}$ is finite, $\mathscr{P}_b$ is finite for every line $b$, and so $\Lambda$ is finite with $|\Lambda| =k \geq 2$.  We have $(\mathscr{P},\mathscr{B},\mathrm{I})$ a Steiner system, and so every point of $(\mathscr{P},\mathscr{B},\mathrm{I})$ is incident with exactly $\displaystyle r =  \frac{v-1}{k-1}$ lines.   It follows by the construction of $\mathfrak{M}(\Gamma,\varphi)$ that every line $z_{p,\mu}$ is incident with exactly $r+1$ points $x_p,y_{b_1,\lambda_1}, \dots, y_{b_r,\lambda_r}$ where $\lambda_i = \varphi(b_i p )^{-1} \cdot \mu$ for each $i \in \{1,\dots,r\}$.  Also, in $\mathfrak{M}(\Gamma,\varphi)$, we have every point of the form $y_{b,\lambda}$ incident with exactly $k$ lines of the form $z_{q_1,\mu_1}, \dots, z_{q_k,\mu_k}$ where $\mu_i = \varphi(bq_i) \cdot \lambda$ for each $i \in \{1,\dots,k\}$.  It is clear by the construction of $\mathfrak{M}(\Gamma,\varphi)$ that a point of type $x_p$ is incident with exactly $|\Lambda| = k$ lines.  
\end{proof}

The result of Theorem 5 is a special case of the result found in \cite{Mcc24}.  We note that in \cite{Mcc24} a construction is obtained from a block design.  Our Construction $\mathfrak{M}$ requires only an incidence gain graph.

In the affine plane example in the next section, the action of $G$ on $\Lambda$ is regular, and so we have the following. 

\begin{proposition}
Let $(\Gamma,\varphi)$ be an incidence gain graph with underlying incidence structure $(\mathscr{P},\mathscr{B},\mathrm{I})$ a linear space, and gain group $G$ acting regularly on a nonempty set $\Lambda$.  Let $p \in \mathscr{P}$ and $b \in \mathscr{B}$ with $p \  \cancel{\mathrm{I}} \  b$, and let $\lambda \in \Lambda$.  We have $\rho_{b,p,\lambda}$ is bijective if and only if $\rho_{b,p}$ is bijective.
\end{proposition}

\begin{proof}
As the action is free, 
\begin{gather*}
q_1 = q_2 \Longrightarrow \rho_{b,p}(q_1) \cdot \lambda = \rho_{b,p}(q_2) \cdot \lambda \\
\text{if and only if} \\
q_1 = q_2 \Longrightarrow \rho_{b,p}(q_1) = \rho_{b,p}(q_2).
\end{gather*}
Thus $\rho_{b,p,\lambda}$ is injective if and only if $\rho_{b,p}$ is injective.

We now show that 
\begin{gather*}
\text{for every } \mu \in \Lambda \text{ there exists } q \ \mathrm{I} \ b \text{ so that } \rho_{b,p}(q) \cdot \lambda = \mu \\
\text{if and only if} \\
\text{for every } g \in G \text{ there exists } q \ \mathrm{I} \ b \text{ so that } \rho_{b,p}(q) = g.
\end{gather*}
Suppose first that $$\text{for every } g \in G \text{ there exists } q \ \mathrm{I} \ b \text{ so that } \rho_{b,p}(q) = g,$$ and let $\mu \in \Lambda$ be arbitrary.  As the action is transitive, there exists $g \in G$ so that $\mu = g \cdot \lambda$.  And so there exists $q$ so that $\rho_{b,p}(q) = g$, and hence $\rho_{b,p}(q) \cdot \lambda = g \cdot \lambda = \mu$.  Conversely, suppose that  $$\text{for every } \mu \in \Lambda \text{ there exists } q \ \mathrm{I} \ b \text{ so that } \rho_{b,p}(q) \cdot \lambda = \mu,$$ and let $g \in G$ be arbitrary.  Let $\mu = g \cdot \lambda$.  And so there exists $q$ so that $\rho_{b,p}(q) \cdot \lambda = \mu = g \cdot \lambda$.  Since the action is free, $\rho_{b,p}(q)  = g $.

Thus $\rho_{b,p,\lambda}$ is surjective if and only if $\rho_{b,p}$ is surjective.
\end{proof}

\section{Affine Plane Examples and Further Inquiry}

Let $\mathbb{F}$ be a field.  An \textit{affine plane over $\mathbb{F}$} is an incidence structure $(\mathscr{P},\mathscr{B},\mathrm{I})$ where the point set $\mathscr{P} = \mathbb{F} \times \mathbb{F}$, and line set $\mathscr{B}$ consists of lines of two types.  A \textit{vertical line}, $L_b$,  we index by $b \in \mathbb{F}$ which is the \textit{$x$-intercept}.  A \textit{non-vertical line}, $L_{m,b}$, we index by $m \in \mathbb{F}$ which is the \textit{slope} and $b \in \mathbb{F}$ which is the \textit{$y$-intercept}.  The incidence $\mathrm{I}$ is given by $(b,y) \  \mathrm{I} \  L_b$ for all $y \in \mathbb{F}$ and $(x,mx+b) \  \mathrm{I} \  L_{m,b}$ for all $x \in \mathbb{F}$.

\begin{theorem}
Suppose $(\mathscr{P},\mathscr{B},\mathrm{I})$ is an affine plane over a field $\mathbb{F}$ and let $\mathbb{F}$ act on itself via addition.  Let $\Gamma = (V,E)$ be the incidence graph of the affine plane, and let us define a gain function $\varphi : \mathscr{F}(E) \rightarrow (\mathbb{F},+)$ as follows:

$$\varphi(e) = 
    \begin{cases}
        -by, & \text{if } e = \{ L_b, (b,y) \}, \\
        xb, & \text{if } e = \{ L_{m,b} , (x,mx+b) \}.
    \end{cases}$$

We have $\mathfrak{M}(\Gamma,\varphi)$ a generalized quadrangle.
\end{theorem}

\begin{proof}
Since the additive action of $\mathbb{F}$ on itself is regular, it follows from Theorem 5 and Proposition 6 that it suffices to show that for every $p \in \mathscr{P}$ and $L \in \mathscr{B}$ so that $p \ \cancel{\mathrm{I}} \ L$, $\rho_{L,p}$ is bijective. 

\textit{Proof of the injectivity of $\rho_{L,p}$.}

Suppose $L \in \mathscr{B}$ and $p \in \mathscr{P}$ so that $p \ \cancel{\mathrm{I}} \ L$.  Let $p_1,p_2$ be distinct points incident with $L$, and let $M_1$ be the unique line incident with $p$ and $p_1$ and let $M_2$ be the unique line incident with $p$ and $p_2$.  Let $w_1 = (L,e_1,p_1,e_2,M_1,e_3,p)$ and let $w_2 = (L,e_6,p_2,e_5,M_2,e_4,p)$ be the corresponding walks in $\Gamma$.  We need to show that $\rho_{L,p}(p_1) \neq \rho_{L,p}(p_2)$.

\textit{Case 1.}  Suppose that $L = L_b$ is a vertical line.  Let $p_1 = (b,y_1)$, $p_2 = (b,y_2)$, and $p = (x,y)$.  And so $x \neq b$ and $y_1 \neq y_2$.

Now $M_1 = L_{m_1,b_1}$ with slope $m_1 = (y-y_1)(x-b)^{-1}$ and $y$-intercept $b_1 = y_1 - m_1 b,$ and $M_2 = L_{m_2,b_2}$ with slope $m_2 = (y-y_2)(x-b)^{-1}$ and $y$-intercept $b_2 = y_2 - m_2 b.$

Hence we need to show that 
\begin{gather}
\nonumber \varphi(e_3) - \varphi(e_2) + \varphi(e_1) \neq  \varphi(e_4) - \varphi(e_5) + \varphi(e_6), \text{  i.e.  } \\
xb_1 - bb_1 - by_1 \neq xb_2 - bb_2 - by_2.
\end{gather}

Substituting for $b_1$, the left hand side of (1) becomes
\begin{gather*}
x(y_1 - (y-y_1)(x-b)^{-1}b) - b(y_1 - (y-y_1)(x-b)^{-1}b) - by_1 = \\
(x-b)y_1 - (x-b)(y-y_1)(x-b)^{-1}b - by_1 = \\
(x-b)y_1 - (y-y_1)b - by_1 = \\
xy_1 - yb - by_1.
\end{gather*}

Similarly, the right hand side of (1) becomes
$$xy_2 - yb - by_2.$$

Hence we need to show that
$$y_1(x - b) \neq y_2(x-b),$$
and this is true since $x \neq b$ and $y_1 \neq y_2$.

\textit{Case 2.}  Suppose that $L = L_{m,b}$ is a non-vertical line.  Let $p_1 = (x_1, y_1)$, $p_2 = (x_2, y_2)$, and $p = (x,y)$.  And so $y \neq mx + b$ and $x_1 \neq x_2$.

If $x=x_1$, then $M_1 = L_x$ is a vertical line, and we have 
\begin{gather*}
\varphi(e_3) - \varphi(e_2) + \varphi(e_1) = \\
-x_1y + xy_1 + x_1b.
\end{gather*}

If $x=x_2$, then $M_2 = L_x$ is a vertical line, and we have
\begin{gather*}
\varphi(e_4) - \varphi(e_5) + \varphi(e_6) = \\
-x_2y + xy_2 + x_2b.
\end{gather*}

If $x \neq x_1$, then $M_1 = L_{m_1,b_1}$ is a non-vertical line with slope $m_1 = (y-y_1)(x-x_1)^{-1}$ and $y$-intercept $b_1 = y_1 - m_1 x_1.$  In this case we have
\begin{gather*}
\varphi(e_3) - \varphi(e_2) + \varphi(e_1) = \\
xb_1 - x_1b_1 + x_1b.
\end{gather*}
Substituting for $b_1$, we obtain
\begin{gather*}
x(y_1 - (y-y_1)(x-x_1)^{-1}x_1) - x_1(y_1 - (y-y_1)(x-x_1)^{-1}x_1) + x_1b = \\
(x-x_1)y_1 - (x-x_1)(y-y_1)(x-x_1)^{-1}x_1 + x_1b = \\
(x-x_1)y_1 - (y-y_1)x_1 + x_1b = \\
xy_1 - x_1y + x_1b.
\end{gather*}

If $x \neq x_2$, then $M_2 = L_{m_2,b_2}$ is a non-vertical line with slope $m_2 = (y-y_2)(x-x_2)^{-1}$ and $y$-intercept $b_2 = y_2 - m_2 x_2.$
Similarly we obtain 
\begin{gather*}
\varphi(e_4) - \varphi(e_5) + \varphi(e_6) = \\
xy_2 - x_2y + x_2b.
\end{gather*}

In any case, we need to show that 
\begin{gather*}
\varphi(e_3) - \varphi(e_2) + \varphi(e_1) \neq \varphi(e_4) - \varphi(e_5) + \varphi(e_6), \text{  i.e.  } \\
xy_1 - x_1y + x_1b \neq xy_2 - x_2y + x_2b.
\end{gather*}

Now $y_1 = mx_1 + b$, and $y_2 = mx_2 + b$, and, after substituting, we need then to show that
$$x_1(mx + b - y) \neq x_2(mx + b - y),$$
and this is true since $y \neq mx + b$ and $x_1 \neq x_2$.

\textit{Proof of the surjectivity of $\rho_{L,p}$.}

\textit{Case 1.}  Suppose that $L = L_b$ is a vertical line and $p = (x,y)$ is not incident with $L$.  And so $x \neq b$.  Let $z \in \mathbb{F}$ be arbitrary.  Let $y_1 = (x-b)^{-1}(z + yb).$  Let $p_1 = (b,y_1)$.  Let $M_1$ be the unique line incident with $p$ and $p_1$, and so $M_1 = L_{m_1,b_1}$ with slope $m_1 = (y-y_1)(x-b)^{-1}$ and $y$-intercept $b_1 = y_1 - m_1 b.$  Let $w_1 = (L,e_1,p_1,e_2,M_1,e_3,p)$ be the corresponding walk in $\Gamma$.  It follows by a similar computation as in the injectivity proof that 
\begin{gather*}
y_1(x-b) - yb = z, \text{  i.e.  } \\
xy_1 - yb - by_1 =z, \text{  i.e.  } \\
xb_1 - bb_1 - by_1 = z, \text{  i.e.  } \\
\varphi(e_3) - \varphi(e_2) + \varphi(e_1) = z.
\end{gather*}

\textit{Case 2.}  Suppose that $L = L_{m,b}$ is a non-vertical line and $p = (x,y)$ is not incident with $L$.  And so $y \neq mx + b$.  Let $z \in \mathbb{F}$ be arbitrary.  Let  $x_1 = (mx + b - y)^{-1} (z - xb).$  Let $p_1 = (x_1, y_1)$ be incident with $L$, and so $y_1 = mx_1 + b$.  Let $M_1$ be the unique line incident with $p$ and $p_1$.  Let $w_1 = (L,e_1,p_1,e_2,M_1,e_3,p)$ be the corresponding walk in $\Gamma$.  

If $x=x_1$, then $M_1 = L_x$ is a vertical line.   It follows by a similar computation as in the injectivity proof that
\begin{gather*}
x_1(mx + b - y) + xb = z, \text{ i.e. } \\
- x_1y + xy_1 + x_1b = z, \text{ i.e. } \\
\varphi(e_3) - \varphi(e_2) + \varphi(e_1) = z.
\end{gather*}

If $x \neq x_1$, then $M_1 = L_{m_1,b_1}$ is a non-vertical line with slope $m_1 = (y-y_1)(x-x_1)^{-1}$ and $y$-intercept $b_1 = y_1 - m_1 x_1.$   It follows by a similar computation as in the injectivity proof that
\begin{gather*}
x_1(mx + b - y) + xb = z, \text{ i.e. } \\
xy_1 - x_1y + x_1b = z, \text{ i.e. } \\
xb_1 - x_1b_1 + x_1b = z, \text{  i.e.  } \\
\varphi(e_3) - \varphi(e_2) + \varphi(e_1) = z.
\qedhere
\end{gather*}
\end{proof}

When $\mathbb{F}$ is finite, say order $q$, we checked by computer that the $GQ(q+1,q-1)$ constructed in Theorem 7 is isomorphic to the dual of the Payne GQ $P(W(q),x)$.  This has been confirmed for values of $q$ up to $16$, and we conjecture that it is always true.  We end with two open problems for further study.

\begin{problem}
Determine all of the gain functions on an affine plane over a field that yield a generalized quadrangle as in Theorem 7.  Do two different such gain functions (on the same affine plane) yield isomorphic generalized quadrangles?  
\end{problem}

\begin{problem}
Find other examples of gain functions on Steiner systems that yield generalized quadrangles.
\end{problem}

\bigskip\noindent {\bf Acknowledgment.} The author thanks Tom Zaslavsky for helpful discussions and suggestions regarding this project.  The author also thanks Anthony D. Berard, Jr.

\bibliographystyle{plainnat}
\bibliography{ref}

\end{document}